\documentclass[journal,final,twocolumn]{IEEEtran}

\usepackage{amssymb,amsmath,amsthm,enumerate,empheq}
\usepackage{subfig}
\usepackage{cite}
\usepackage{graphicx}
\usepackage[usenames,dvipsnames,svgnames,table]{xcolor}

\graphicspath{{Figures/}}

\theoremstyle{definition}
\newtheorem{defn}{Definition}

\newtheorem{rem}{Remark}

\newtheorem{thm}{Theorem}
\newtheorem{ex}{Example}
\newtheorem{cor}{Corollary}

\newcommand{\norm}[1]{\Vert#1\Vert}

\newcommand{\abs}[1]{\left\vert#1\right\vert}
\newcommand{\set}[1]{\left\{#1\right\}}
\newcommand{\Real}{\mathbb R}
\newcommand{\eps}{\varepsilon}
\renewcommand{\phi}{\varphi}

\renewcommand{\subset}{\subseteq}

\newcommand{\dom}{\text{dom}\,}

\newcommand{\F}{\mathcal{F}}
\newcommand{\G}{\mathcal{G}}
\newcommand{\M}{\mathcal{M}}

\newcommand{\C}{\mathcal{C}}
\newcommand{\W}{\mathcal{W}}
\newcommand{\K}{\mathcal{K}}
\newcommand{\KL}{\mathcal{KL}}
\newcommand{\D}{\mathcal{D}}

\newcommand{\A}{\mathcal{A}}
\newcommand{\AD}{\mathcal{A}^{\Delta}}
\newcommand{\Atj}{\mathcal{A}_{[t,j]}^{\Delta}}
\newcommand{\HM}{\mathcal{H_M}}

\newcommand{\ra}{\rightarrow}
\newcommand{\raa}{\rightrightarrows}

\newcommand{\Z}{\mathbb{Z}}

\newcommand{\MD}{\mathcal{M}^{\Delta}}

\newcommand{\HMD}{\mathcal{H}_{\mathcal{M}}^{\Delta}}
\newcommand{\SHMD}{\mathcal{S}_{\mathcal{H}_\mathcal{M}^{\Delta}}}

\begin{document}

\title{\Large\bfseries Lyapunov-based sufficient conditions for stability of hybrid systems with memory}

\author{Jun Liu and Andrew R. Teel
\thanks{This work is supported, in part, by Royal Society grant IE130106, EU FP7 grant PCIG13-GA-2013-617377, US AFOSR grants FA9550-12-1-0127 and FA9550-15-1-0155, and US NSF grant ECCS-1232035.}
\thanks{J. Liu is currently with the Department of Automatic Control and Systems Engineering, University of Sheffield, Sheffield S1 3JD, United Kingdom. As of September 2015, he is with the Department of Applied Mathematics, University of Waterloo, Waterloo, Ontario N2L 3G1, Canada (e-mail: j.liu@uwaterloo.ca).}
\thanks{A.~R.~Teel is with the Center for Control, Dynamical Systems, and Computation and the Department of Electrical and Computer Engineering, University of California, Santa Barbara, CA 93106, USA (e-mail: teel@ece.ucsb.edu).}}

\maketitle

\begin{abstract}

Hybrid systems with memory are dynamical systems exhibiting both hybrid and delay phenomena.
In this note, we study the asymptotic stability of hybrid systems with memory using generalized concepts of solutions. These generalized solutions, motivated by studying robustness and well-posedness of such systems, are defined on hybrid time domains and parameterized by both continuous and discrete time. We establish Lyapunov-based sufficient conditions for asymptotic stability using both Lyapunov-Razumikhin functions and Lyapunov-Krasovskii functionals. Examples are provided to illustrate these conditions.

\end{abstract}

\begin{IEEEkeywords}

Hybrid systems, time delay, functional inclusions, generalized solutions, stability.

\end{IEEEkeywords}

\IEEEpeerreviewmaketitle

\section{Introduction}

While delay phenomena and hybrid dynamics are ubiquitous in both nature and engineering applications, the interplay between these two is particularly pronounced in control systems, where the use of hybrid control algorithms has increasingly gained popularity and delays are often inevitable, e.g., in situations where the control loop is closed over a network.

Hybrid systems with memory refer to dynamical systems exhibiting both hybrid and delay phenomena. Such systems have attracted considerable attention and asymptotic stability for hybrid systems with delays has been addressed in the past in various settings (see, e.g., \cite{chen2009input,liu2001uniform,liu2006stability,liu2011generalized,liu2011input,yan2008stability,yuan2003uniform}). All these results, however, have been established using the classical notion of solutions, which are typically considered to be piecewise continuous and parameterized by the continuous time alone. Meanwhile, generalized solutions of hybrid inclusions, defined on hybrid time domains and parameterized by both continuous and discrete times, have been proposed and proven effective for the robust stability analysis of hybrid systems without memory \cite{goebel2004hybrid,goebel2012hybrid,sanfelice2008generalized}.

Recent work in \cite{liu2012generalized,liu2014hybrid-ifac} proposed a framework that allows to study hybrid systems with delays through generalized solutions. One of the major different ideas in \cite{liu2012generalized,liu2014hybrid-ifac} is to consider a phase space equipped with  the graphical convergence topology, instead of using the conventional choice of phase space of piecewise continuous functions equipped with uniform convergence topology. The latter is not well-suited to handle discontinuities caused by jumps in hybrid systems, especially when structural properties of the solutions are concerned. By using tools from functional differential inclusions, basic existence and well-posedness results have been established \cite{liu2014hybrid-siam}. As a consequence of well-posedness, it is also proved in \cite{liu2014hybrid-siam} that $\KL$ pre-asymptotic stability is robust for well-posed hybrid systems with memory.

The main contribution of this note is to provide Lyapunov-based sufficient conditions for studying asymptotic stability of hybrid systems with memory using generalized solutions. We prove two sets of such conditions, one using Lyapunov-Razumikhin functions and the other using Lyapunov-Krasovskii functionals. These results extend to the hybrid setting the classical stability results for functional differential equations \cite{hale1993introduction} and more recent results on delay difference equations or inclusions (e.g., \cite{liu2007razumikhin,gielen2013tractable,gielen2013necessary}). They also extend some preliminary results on this topic found in \cite{liu2014hybrid-ifac}.

\section{Preliminaries}\label{sec:pre}

\emph{Notation:} $\Real^n$ denotes the $n$-dimensional Euclidean space with its norm denoted by  $\abs{\cdot}$; $\mathbb{Z}$ denotes the set of all integers; $\Real_{\ge 0}= [0,\infty)$, $\Real_{\le 0}=(-\infty,0]$, $\Z_{\ge 0}=\set{0,\,1,\,2,\,\cdots}$, and $\Z_{\le 0}=\set{0,\,-1,\,-2,\,\cdots}$. A continuous function $\alpha:\,\Real_{\ge 0} \rightarrow \Real_{\ge 0}$ is said to belong to class $\mathcal{K}$ if it is strictly increasing and satisfies $\alpha(0) = 0$. A class $\mathcal{K}$ function is said to belong to class $\mathcal{K}_{\infty}$ if it further satisfies $\lim_{s \rightarrow \infty} \alpha(s) = \infty$. A continuous function $\beta:\,\Real_{\ge 0}\times\Real_{\ge 0} \rightarrow \Real_{\ge 0}$ is said to belong to class $\mathcal{KL}$ if, for each fixed $s$, the function $r\mapsto\beta(r,s)$ belongs to class $\mathcal{K}$ and, for each fixed $r$, the function $s\mapsto\beta(r,s)$ is decreasing with respect to $s$ and satisfies $\beta(r,s) \rightarrow 0$ as $s \rightarrow \infty$.

\subsection{Hybrid systems with memory}

We start with the definition of hybrid time domains and hybrid arcs
\cite{goebel2012hybrid,goebel2006solutions} for hybrid systems and generalize them in order to define hybrid systems with memory. The following definitions first appeared in \cite{liu2012generalized,liu2014hybrid-ifac}.

\begin{defn}
Consider a subset $E\subset \Real\times \Z$ with $E=E_{\ge 0}\cup E_{\le 0}$, where $E_{\ge 0}:=(\Real_{\ge 0}\times \Z_{\ge 0})\cap E$ and $E_{\le 0}:=(\Real_{\le 0}\times \Z_{\le 0})\cap E$. The set $E$ is called a \emph{compact hybrid time domain with memory} if
$
E_{\ge 0}=\bigcup_{j=0}^{J-1}([t_j,t_{j+1}],j)
$
and
$
E_{\le 0}=\bigcup_{k=1}^{K}([s_k,s_{k-1}],-k+1)
$
for some finite sequence of times $s_{K}\le\cdots\le s_1\le s_0=0=t_0\le t_1\le \cdots\le t_J$. It is called a \emph{hybrid time domain with memory} if, for all $(T,J)\in E_{\ge 0}$ and all $(S,K)\in \Real_{\ge 0}\times \Z_{\ge 0}$,
$
(E_{\ge 0}\cap ([0,T]\times\set{0,\,1,\,\cdots,\,J})) \cup (E_{\le 0}\cap ([-S,0]\times\set{-K,\,-K+1,\,\cdots,\,0}))
$
is a compact hybrid time domain with memory. The set $E_{\le 0}$ is called a \emph{hybrid memory domain}.
\end{defn}

\begin{defn}
A \emph{hybrid arc with memory} consists of a hybrid time domain with memory, denoted by $\dom x$, and a function $x:\,\dom x\ra \Real^n$ such that $x(\cdot,j)$ is locally absolutely continuous on $I^{j}=\set{t:\,(t,j)\in\dom\,x}$ for each $j\in\Z$ such that $I^{j}$ has nonempty interior. In particular, a hybrid arc $x$ with memory is called a \emph{hybrid memory arc} if $\dom x\subset\Real_{\le 0}\times\Z_{\le 0}$. We shall simply use the term \emph{hybrid arc} if we do not have to distinguish between the above two hybrid arcs. We write $\dom_{\ge 0}(x):=\dom x\cap(\Real_{\ge 0}\times \Z_{\ge 0})$ and  $\dom_{\le 0}(x):=\dom x\cap(\Real_{\le 0}\times \Z_{\le 0})$.
\end{defn}

We shall use $\M$ to denote the collection of all hybrid memory arcs. Moreover, given $\Delta\in[0,\infty)$, we denote by $\MD$ the collection of hybrid memory arcs $\phi$ satisfying the following two conditions: (1) $s+k\ge -\Delta-1$ for all $(s,k)\in\dom\phi$; and (2) there exists $(s',k')\in \dom\phi$ such that $s'+k'\le-\Delta$. The constant $\Delta$ roughly captures the size of memory for a hybrid system. The above two conditions ensure that the memory size is at least $\Delta$ and at most $\Delta+1$. We allow this variability in order to capture certain graphical convergence properties of solutions related to the well-posedness and robustness of hybrid systems with memory (see \cite{liu2014hybrid-siam} for more details). 

Given a hybrid arc $x$, we define an operator $\AD_{[\cdot,\cdot]}x:\,\dom_{\ge 0}(x)\ra \MD$ by
$
\AD_{[t,j]}x(s,k)=x(t+s,j+k),
$
for all $(s,k)\in\dom (\AD_{[t,j]}x)$, where $(t,j)\in\dom_{\ge 0}(x)$ and
\begin{align*}
\dom (\AD_{[t,j]} x)&:=\Big\{(s,k)\in \Real_{\le 0}\times \Z_{\le 0}:\\
&\qquad\,(t+s,j+k)\in\dom x,\;s+k\ge -\Delta_{\inf}\Big\},
\end{align*}
$$
\Delta_{\inf}:=\inf\Big\{\delta\ge\Delta:\,\exists(t+s,j+k)\in\dom x\;\text{s.t.}\;   s+k=-\delta\Big\}.
$$
It follows that if $\AD_{[0,0]}x\in \MD$, then $\AD_{[t,j]}x\in \MD$ for any $(t,j)\in \dom_{\ge 0}(x)$.

\begin{defn}
A \emph{hybrid system with memory of size $\Delta$} is defined by a 4-tuple $\HMD=(\C,\F,\D,\G)$:
\begin{itemize}
\item a set $\C\subset\MD$, called the \emph{flow set};
\item a set-valued functional $\F:\MD\raa \Real^n$, called the \emph{flow map};
\item a set $\D\subset\MD$, called the \emph{jump set};
\item a set-valued functional $\G:\MD\raa \Real^n$, called the \emph{jump map}.
\end{itemize}
\end{defn}

Given a hybrid memory arc $\phi\in\MD$ and $g\in\Real^n$, we define $\phi_g^+$ be a hybrid memory arc in $\MD$ satisfying $\phi_g^+(0,0)  = g$ and $\phi_g^+(s,k-1) = \phi(s,k)$ for all $(s,k)\in\dom \phi$. Furthermore, we define $\G^+(\D):=\set{\phi_g:\,g\in\G(\phi),\;\phi\in \D}$. Intuitively, $\phi_g^+$ is the hybrid memory arc following $\phi$ after taking a jump of value $g$; $\G^+(\D)$ is the set of hybrid memory arcs that can be obtained by applying the functional $\G$ on the jump set $\D$.

\begin{defn}\label{def:sol}
 A hybrid arc is a \emph{solution to the hybrid system $\HMD$} if the initial data $\AD_{[0,0]}x\in \C\cup \D$ and:
 \begin{enumerate}[(S1)]
 \item for all $j\in \Z_{\ge 0}$ and almost all $t$ such that $(t,j)\in\dom_{\ge 0} (x)$,
 \begin{equation}
 \Atj x\in \C,\quad \dot{x}(t,j)\in \F(\Atj x),
 \end{equation}
 \item for all $(t,j)\in\dom_{\ge 0} (x)$ such that $(t,j+1)\in\dom_{\ge 0} (x)$,
 \begin{equation}
 \Atj x\in \D,\quad x(t,j+1)\in \G(\Atj x).
 \end{equation}
 \end{enumerate}
The solution $x$ is called \emph{nontrivial} if $\dom_{\ge 0}(x)$ has at least two points. It is called \emph{complete} if $\dom_{\ge 0}(x)$ is unbounded. It is called \emph{maximal} if there does not exist another solution $y$ to $\HMD$ such that $\dom x$ is a proper subset of $\dom y$ and $x(t,j)=y(t,j)$ for all $(t,j)\in\dom x$. The set of all maximal solutions to $\HMD$ starting from some initial data $\phi\in\MD$ is denoted by $\SHMD(\phi)$.
\end{defn}

We refer the readers to \cite{liu2012generalized,liu2014hybrid-ifac,liu2014hybrid-siam} for existence of generalized solutions and well-posedness for hybrid systems with memory. The main results of this paper are on $\KL$ pre-asymptotic stability with respect to a closed set in $\Real^n$ for hybrid systems with memory. The definition for $\KL$ pre-asymptotic stability is given below.

\begin{defn}\label{def:stability}
Let $\W\subset\Real^n$ be a closed set. The set $\W$ is said to be \emph{$\KL$ pre-asymptotically stable} for $\HMD$ if there exists a $\KL$ function $\beta$ such that any solution $x$ to $\HMD$ satisfies
\begin{align*}
\abs{x(t,j)}_{\W}&\le \beta(\norm{\AD_{[0,0]}x}_{\W},t+j),\quad\forall (t,j)\in\dom_{\ge 0}(x),
\end{align*}
where $\abs{z}_{\W}:=\inf_{y\in\W}\abs{y-z}$ for $z\in\Real^n$ and
$\norm{\phi}_{\W}:=\sup_{(s,k)\in\dom\phi\atop s+k\ge -\Delta-1}\abs{\phi(s,k)}_{\W}$
for $\phi\in\MD$.
\end{defn}

\section{Lyapunov conditions for asymptotic stability}\label{sec:stability}

In this section, we present the main results on Lyapunov-based sufficient conditions for the asymptotic stability of hybrid systems with memory. We provide two sets of conditions, one in terms of Lyapunov-Razumikhin functions and the other using Lyapunov--Krasovskii functionals.

\subsection{Sufficient conditions by Lyapunov-Razumikhin functions}

Razumikhin theorems \cite{razumikhin1956stability} have been a very useful tool for the stability analysis of delay or functional differential equations (see, e.g., \cite[Theorem 4.2, Chapter 5]{hale1993introduction}). Typical stability criteria in a Razumikhin-type theorem involve a positive definite Lyapunov function whose derivative along solutions is negative definite only if the current value of the Lyapunov function exceeds certain thresholds relative to the past values of the Lyapunov function over a delay period. 

The following result, which extends a preliminary result found in \cite{liu2014hybrid-ifac}, gives a general Razumikhin-type theorem for hybrid systems with memory.

\begin{thm}\label{thm:stability}
Let $\HMD=(\C,\F,\D,\G)$ be a hybrid system with finite memory (i.e., $\Delta<\infty$) and let $\W\subset\Real^n$ be a closed set. If there exists a continuously differentiable function $V:\,\Real^n\ra\Real_{\ge 0}$, $\K_\infty$ functions $\alpha_i$ ($i=1,2$), a positive definite and continuous function $\alpha_3:\,\Real_{\ge 0}\ra\Real_{\ge 0}$, and continuous functions $p:\,\Real_{\ge 0}\ra\Real_{\ge 0}$ and $\rho:\,\Real_{\ge 0}\ra\Real_{\ge 0}$ with $p(r)>r$ and $\rho(r)<r$ for all $r>0$ such that the following hold:
\begin{itemize}
\item[(i)] $\alpha_1(\abs{\phi(0,0)}_{\W})\le V(\phi(0,0))\le \alpha_2(\abs{\phi(0,0)}_{\W})$ for all $\phi\in \C\cup \D\cup\G^+(\D)$;
\item[(ii)] $\nabla V(\phi(0,0))\cdot f\le -\alpha_3 (V(\phi(0,0)))$ for all $\phi\in \C$ such that $p(V(\phi(0,0)))\ge \overline{V}(\phi)$ and all $f\in\F(\phi)$;
\item[(iii)] $V(g)\le \rho(\overline{V}(\phi))$ for all $\phi\in\D$ and all $g\in\G(\phi)$,
\end{itemize}
where $\overline{V}(\phi):=\max_{(s,k)\in\dom\phi\atop s+k\ge -\Delta-1} V(\phi(s,k))$, then $\W$ is $\KL$ pre-asymptotically stable for $\HMD$.
\end{thm}

\begin{proof}
Let $x$ be a solution to $\HMD$. We first show that $\overline{V}(\AD_{[t,j]}x)$ is non-increasing for $(t,j)\in\dom x$. If both $(t,j)\in\dom x$ and $(t,j+1)\in\dom x$, we have
$
V(x(t,j+1))\le \rho(\overline{V}(\AD_{[t,j]}x)\le \overline{V}(\AD_{[t,j]}x),
$
which implies $\overline{V}(\AD_{[t,j+1]}x)\le \overline{V}(\AD_{[t,j]}x)$. If $(t+s,j)\in\dom x$ for all $s\in [0,h]$ for some small $h>0$, we consider two cases: (a) $V(x(t,j))=\overline{V}(\AD_{[t,j]}x)$, and (b) $V(x(t,j))<\overline{V}(\AD_{[t,j]}x)$. If (a) holds, we claim that $V(x(t+s,j))\le V(x(t,j))$ for all $s\in [0,h]$. This claim would imply $\overline{V}(\AD_{[t+s,j]}x)\le \overline{V}(\AD_{[t,j]}x)$ for all $s\in [0,h]$. We prove this by showing that, for any fixed $\eps>0$, $V(x(t+s,j))<V(x(t,j))+\eps$ for all $s\in [0,h]$. Suppose this is not the case. Then let $\bar{s}:=\inf\set{s\in [0,h]:\,V(x(t+s,j))\ge V(x(t,j))+\eps}$. It follows that $V(x(t+\bar{s},j))=V(x(t,j))+\eps$ and $V(x(t+s,j))<V(x(t,j))+\eps$ for all $s\in [0,\bar{s})$. Using continuity of $V(x(t+s,j))$ with respect to $s\in[0,h]$, there exists $\underline{s}\in[0,\bar{s}]$ such that $V(x(t+s,j))\ge V(x(t,j))+\frac{\eps}{2}$ for all $s\in [\underline{s},\bar{s}]$. Let $$\eta=\min_{u\in[V(x(t,j))+\frac{\eps}{2},V(x(t,j))+\eps]}\set{p(u)-u}.$$
Then $\eta>0$. At $s=\bar{s}$, we have
\begin{align*}
p(V(x(t+\bar{s},j)))&=p(V(x(t,j))+\eps)\\
&\ge V(x(t,j))+\eps+\eta.
\end{align*}
It follows from the continuity of $p(V(x(t+s,j)))$ with respect to $s$ that there exists some small $\delta>0$ such that $\bar{s}-\delta\in[\underline{s},\bar{s})$ and
\begin{align*}
p(V(x(t+s,j)))&\ge V(x(t,j))+\eps+\frac{\eta}{2}\\
&> \overline{V}(\AD_{[t+s,j]}x),\quad\forall s\in [\bar{s}-\delta,\bar{s}].
\end{align*}
Condition (ii) of the theorem implies that
$$
\frac{dV(x(t+s,j))}{dt}\le -\alpha_3(V(x(t+s,j)))<0,
$$
for almost all $s\in [\bar{s}-\delta,\bar{s}]$. It follows that $V(x(t+\bar{s},j))<V(x(t+\bar{s}-\delta,j))<V(x(t,j))+\eps$, which contradicts the definition of $\bar{s}$. If (b) holds, it follows from the continuity of $x(t+s,j)$ on $[0,h]$ that, if $h$ is sufficiently small, then $\overline{V}(\AD_{[t+s,j]}x)\le \overline{V}(\AD_{[t,j]}x)$ for all $s\in [0,h]$. We have proved that $\overline{V}(\AD_{[t,j]}x)$ is non-increasing for $(t,j)\in\dom x$.

Now consider any solution $x$ to $\HMD$. It follows from condition (i) that
\begin{equation}\label{eq:stability}
V(x(t,j))\le \overline{V}(\AD_{[t,j]}x)\le \overline{V}(\AD_{[0,0]}x) \le \alpha_2(\norm{\AD_{[0,0]}x}_{\W}),
\end{equation}
for all $(t,j)\in\dom_{\ge 0}(x)$.

Fix any $\eta>0$ and consider any solution $x$ to $\HMD$ with $\norm{\AD_{[0,0]}x}_{\W}\le \eta$. It follows from (\ref{eq:stability}) that $V(x(t,j))\le \alpha_2(\eta)$ for all $(t,j)\in\dom_{\ge 0}(x)$. We further show that for all $\eps>0$, there exists some $T>0$ such that $V(x(t,j))\le \eps$ for all $(t,j)\in\dom x$ with $t+j\ge T$. Without loss of generality, assume $\eps<\alpha_2(\eta)$. Define
$$
\gamma:=\min\set{\min_{s\in [\eps,\alpha_2(\eta)]}\alpha_3(\eps),\min_{s\in [\eps,\alpha_2(\eta)]}\set{s-\rho(s)}}
$$
and
$$
a:=\min\set{\min_{s\in [\eps,\alpha_2(\eta)]}\set{p(s)-s},\min_{s\in [\eps,\alpha_2(\eta)]}\set{s-\rho(s)}}.
$$
Let $N$ be the smallest integer such that $\eps+Na\ge \alpha_2(\eta)$. Both $\gamma$ and $a$ are positive, because $\eps>0$, $p(s)>s$, and $\rho(s)<s$ for all $s>0$.

We claim that there exists $T_1$ such that $V(x(t,j))\le \eps+(N-1)a$ for all $(t,j)\in\dom x$ with $t+j\ge T_1$. We prove this in two steps.
\begin{itemize}
\item[(A)] First, we show that  $V(x(t,j))\le \eps+(N-1)a$ holds for some $(t,j)\in\dom x$ and $t+j\ge 0$.
\item[(B)] Second, we show that once $V(x(t,j))\le \eps+(N-1)a$ for some $(t,j)\in\dom x$ and $t+j\ge 0$, this holds for all $(t,j)\in\dom x$ beyond this instant.
\end{itemize}
Suppose (A) is not true, then we have $V(x(t,j))>\eps+(N-1)a$ holds for all $(t,j)\in\dom x$ and $t+j\ge 0$. It follows that for all $(t,j)\in \dom x$ with $t+j\ge 0$, we have
$$
p(V(x(t,j)))\ge V(x(t,j))+a > \eps+Na \ge \alpha_2(\eta)= \overline{V}(\AD_{[t,j]}x).
$$
If $I^j:=\set{t:\,(t,j)\in \dom x}$ has non-empty interior, we have from condition (ii) that
\begin{equation}\label{eq:decay2}
\frac{dV(x(t,j))}{dt} \le - \alpha_3(V(x(t,j))) \le -\gamma.
\end{equation}
If both $(t,j)\in\dom x$ and $(t,j+1)\in\dom x$, we have from condition (iii) that
\begin{equation}\label{eq:decay1}
V(x(t,j+1)) - V(x(t,j)) \le \rho(\overline{V}(\AD_{[t,j]}x))-V(x(t,j))\le -\gamma.
\end{equation}

Combining (\ref{eq:decay1}) and (\ref{eq:decay2}) gives
$
V(x(t,j))\le V(x(0,0)) -\gamma(t+j),
$
which holds for all $(t,j)\in\dom x$ with $t+j\ge 0$. This would lead to a contradiction if $t+j$ is sufficiently large.

Suppose (B) is not true. Then there exists $(t',j')\in \dom x$ such that
$$
t'+j'=\inf\set{s+k\ge t+j:\,V(x(s,k))>\eps+(N-1)a}.
$$
We consider two cases.

If $V(x(t',j'))>\eps+(N-1)a$, it must be that $(t',j'-1)\in\dom x$ and $V(x(t',j'-1))\le \eps+(N-1)a$. This would lead to a contradiction that
$$
V(x(t',j'))\le \rho(\overline{V}(\AD_{[t',j'-1]}x)) \le \eps+(N-1)a,
$$
which holds in either case of $\overline{V}(\AD_{[t',j'-1]}x)\ge \eps+(N-1)a$ or $\overline{V}(\AD_{[t',j'-1]}x)\le \eps+(N-1)a$.
Indeed, if the former holds, then $\rho(\overline{V}(\AD_{[t',j'-1]}x))\le \overline{V}(\AD_{[t',j'-1]}x)-a \le \eps+(N-1)a$; if the latter holds, then $\rho(\overline{V}(\AD_{[t',j'-1]}x))\le \overline{V}(\AD_{[t',j'-1]}x) \le \eps+(N-1)a$.

If $V(x(t',j'))=\eps+(N-1)a$, it must hold that $(t'+s,j')\in\dom x$ for all $s\in [0,h]$ for some small $h>0$. Since $p(V(x(t',j')))\ge V(x(t',j'))\ge V(x(t',j'))+a  = \eps+Na\ge \overline{V}(\AD_{[t',j']}x)$, it follows from condition (ii) that
$
\frac{dV(x(t',j'))}{dt}\le -\alpha_3(V(x(t',j')))<0,
$
which contradicts the definition of $(t',j')$.

Combining (A) and (B) above leads to
$
\overline{V}(\AD_{[t,j]}x)\le \eps+(N-1)a,
$
for all $t+j\ge T_1+\Delta+1$.

Repeating the same argument above, we can inductively show that, for $1\le k\le N$, there exists $T_k$ such that
$
\overline{V}(\AD_{[t,j]}x)\le \eps+(N-k)a,
$
for all $t+j\ge T_k+\Delta+1$. Taking $k=N$ and $T=T_N+\Delta+1$, we have established
\begin{equation}\label{eq:attraction}
V(x(t,j))\le\overline{V}(\AD_{[t,j]}x)\le \eps,
\end{equation}
for all $t+j\ge T$. Note that the choice of $T$ only depends on $\eps$ and $\eta$. Thus, combining (\ref{eq:stability}) and (\ref{eq:attraction}), we have proved that $\W$ is $\KL$ pre-asymptotically stable for $\HMD$. The existence of a $\KL$ estimate follows from standard arguments (see, e.g., \cite[Appendix C.6]{khalil1996nonlinear}.
\end{proof}

\begin{rem}
Theorem \ref{thm:stability} extends to the hybrid setting the classical stability results for functional differential equations \cite[Theorem 4.2, Chapter 5]{hale1993introduction} and more recent results on delay difference equations or inclusions (e.g., \cite{liu2007razumikhin,gielen2013tractable,gielen2013necessary}). More specifically, the proof techniques for the evolution of solutions by the flow map resemble that for classical stability results for functional differential equations. The key differences here include condition (iii), which is different from and more general than the conditions proposed for functional difference equations or inclusions (e.g., \cite{liu2007razumikhin,gielen2013tractable,gielen2013necessary}), and arguments to deal with the hybrid nature of the system, which allows the system to have multiple consecutive jumps.
\end{rem}

The following result, which first appeared in \cite{liu2014hybrid-ifac}, establishes Halanay-type inequalities \cite{halanay1966differential} for hybrid systems. It can be proved as an immediate corollary of Theorem \ref{thm:stability}.

\begin{cor}\label{cor:stability}\cite{liu2014hybrid-ifac}
Let $\HMD=(\C,\F,\D,\G)$ be a hybrid system with memory and let $\W\subset\Real^n$ be a closed set. If there exists a continuously differentiable function $V:\,\Real^n\ra\Real_{\ge 0}$, $\K_\infty$ functions $\alpha_i$ ($i=1,2$), and positive constants $\mu>q$ and $\rho<1$ such that the following hold:
\begin{itemize}
\item[(i)] $\alpha_1(\abs{\phi(0,0)}_{\W})\le V(\phi(0,0))\le \alpha_2(\abs{\phi(0,0)}_{\W})$ for all $\phi\in \C\cup \D\cup\G^+(\D)$;
\item[(ii)] $\nabla V(\phi(0,0))\cdot f\le -\mu V(\phi(0,0))+q\overline{V}(\phi)$ for all $\phi\in \C$ and $f\in\F(\phi)$;
\item[(iii)] $V(g)\le \rho\overline{V}(\phi)$ for all $\phi\in\D$ and $g\in\G(\phi)$,
\end{itemize}
where $\overline{V}(\phi)=\max_{-\Delta-1\le s+k\le 0} V(\phi(s,k))$, then $\W$ is $\KL$ pre-asymptotically stable for $\HMD$.
\end{cor}

We use the following example (modified from Example 3.12 of \cite{goebel2012hybrid}) to illustrate how the results established in Theorem \ref{thm:stability} and Corollary \ref{cor:stability} can be applied.

\begin{ex}[Sampled-data systems with delayed measurements]\label{ex1}
Consider a linear sampled-data system
\begin{align}
        &\left\{\begin{aligned}
        \dot{z} & =Az+Bu\\
        \dot{u} & =0\\
        \dot{\tau} & = 1
        \end{aligned}\right\},\quad z\in\Real^n,\quad u\in\Real^m,\quad \tau\in [0,\delta],\\
        &\left\{\begin{aligned}
        z^+ & =z\\
        u^+ & =K\hat{z}\\
        \tau^+ & = 0
        \end{aligned}\right\},\quad z\in\Real^n,\quad u\in\Real^m,\quad \tau\in \set{\delta},
\label{eq:sample}
\end{align}
where $\hat{z}$ is the delayed measurement affected by some fixed sampling delay $r>0$ and $\delta>0$ is the fixed sampling period. The system corresponds to a hybrid system $\HMD=(\C,\F,\D,\G)$ with
$$
\F(\psi):=\begin{bmatrix}A\phi_z(0,0)+B\phi_u(0,0)\\
0\\
1
\end{bmatrix},
$$
$$
\G(\phi):=\begin{bmatrix}\phi_z(0,0)\\
\cup_{\{k\in\Z_{\le 0}:\,(-r,k)\in\dom\psi\}}K\phi_z(-r,k)\\
0
\end{bmatrix},
$$
$$
\C:=\set{\psi=(\phi_z,\phi_u,\tau)\in\MD:\,\tau(0,0)\in [0,\delta]},
$$
$$
\D:=\set{\psi=(\phi_z,\phi_u,\tau)\in\MD:\,\tau(0,0)=\delta},
$$
where $\psi=(\phi,\tau)\in\MD$ and $\phi=(\phi_z,\phi_u)$. 

Let $\W:=\set{0}\times\set{0}\times[0,\delta]$. Since $\psi=(\phi,\tau)$, $\phi=(\phi_z,\phi_u)$, and $\tau\in [0,\delta]$, we have, by the definitions of $\abs{\cdot}_{\W}$ and $\norm{\cdot}_{\W}$ (see Definition \ref{def:stability}), that $\abs{\psi(0,0)}_{\W}=\abs{\phi(0,0)}$ and $\norm{\psi}_{\W}=\norm{\phi}$ for all $\psi=(\phi,\tau)\in\C\cup\D\subset\MD$, where $\norm{\phi}=\sup_{(t,j)\in\dom\phi\atop -\Delta-1\le t+j\le 0}\abs{\phi(t,j)}$.
\end{ex}

Let $x=(z,u,\tau)\in\Real^{n+m+1}$, $x_1=(z,u)$, and $x_2=\tau$. Consider a Lyapunov function candidate of the form $V(x):=e^{-\sigma x_2}V_1(x)$ ($\sigma>0$) with
$
V_1(x):=W(\exp(A_f(\delta-x_2))x_1)
$
and $W(x_1):=x_1^TPx_1$, where $A_f:=\begin{bmatrix}A & B\\
0 & 0
\end{bmatrix}$ and $P$ is a positive definite symmetric matrix such that $H^TPH-P$ is negative definite with $H:=\exp(A_f\delta)A_g$, where $A_g:=\begin{bmatrix}I & 0\\
K & 0
\end{bmatrix}$. Note that the matrix $H$ captures the evolution of the variable $x_1$ at sampling times and just before jumps in the delay-free case \cite{goebel2012hybrid}. By this assumption, we know that there exists $\rho<1$ such that $W(Hx_1)\le \rho W(x_1)$.

It can be easily verified that there exist positive constants $c_1$ and $c_2$ such that
$
c_1\abs{\psi(0,0)}_\W^2\le V(\psi(0,0))\le c_2 \abs{\psi(0,0)}_\W^2.
$
In other words, condition (i) of Corollary \ref{cor:stability} is verified. We now verify condition (ii). It can be shown that
$\nabla V(\psi(0,0))\cdot \F(\psi) = -\sigma V(\psi(0,0)),$ for all $\psi\in \C$. This is because $\nabla V_1(\psi(0,0))\cdot \F(\psi) =0$. To check  condition (iii), one can verify that, for all $\psi\in\D$ and $g\in\G (\psi)$,
\begin{align*}
V(g) & \le \rho e^{\sigma \delta} V(\psi(0,0))+ M r e^{\sigma \delta} \overline{V}(\psi) \le \hat{\rho}\overline{V}(\psi),
\end{align*}
where $\hat{\rho}:=\rho e^{\sigma \delta}+ M r e^{\sigma \delta}$ and $M$ is some positive constant. In other words, if both $r$ and $\delta$ are sufficiently small, then $\hat{\rho}<1$ and $\W$ is $\KL$ pre-asymptotically stable for $\HMD$. Fig. \ref{fig:sample} shows some time-domain numerical simulations of $\HMD$ satisfying the above conditions with 
$A=\begin{bmatrix}4 & 1\\
5 & -3
\end{bmatrix}$, $B=\begin{bmatrix}-3\\-2\end{bmatrix}$, $K=[4\;-2]$, $\delta=0.2$, and $r=0.01$.

\begin{figure}[ht!]
\centering
    \includegraphics[width=0.4\textwidth]{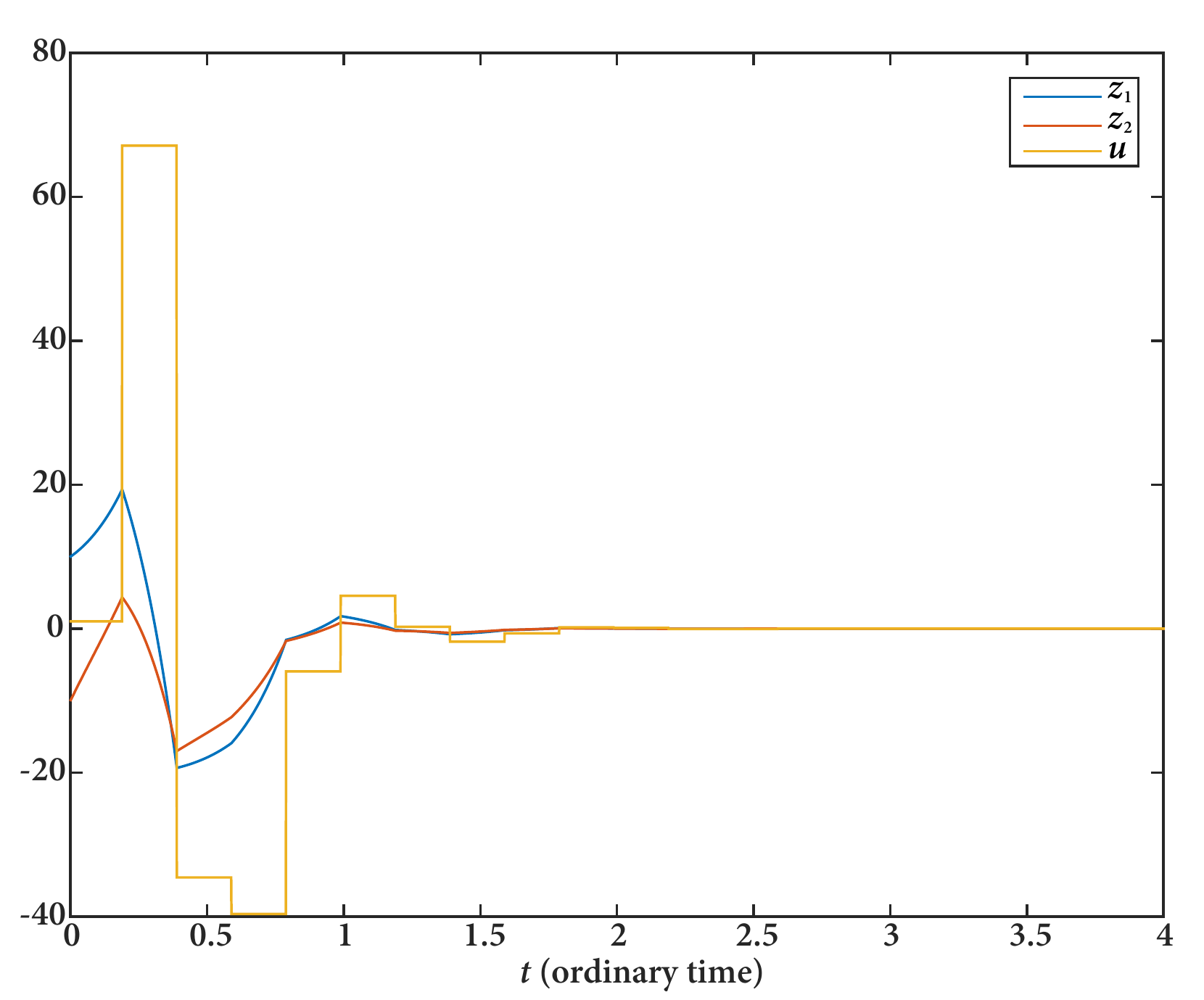}
  \caption{Simulation results for Example \ref{ex1}.}\label{fig:sample}
\end{figure}

\subsection{Sufficient conditions by Lyapunov--Krasovskii functionals}

As in functional differential equations, Lyapunov functionals can be used to formulate sufficient conditions for analyzing stability of hybrid systems with memory. Given a functional $V:\,\MD\ra\Real_{\ge 0}$, we define the upper right-hand derivative of $V$ at $\phi\in\MD$ along the solutions of $\HMD$ as follows:
\begin{equation*}
D^+V(\phi):=\sup_{x\in\SHMD(\phi)}\limsup_{h\rightarrow 0^+}\frac{V(\mathcal{A}_{[h,0]}^{\Delta}x)-V(\mathcal{A}_{[0,0]}^{\Delta}x)}{h}.
\end{equation*} 
The following result provides a set of such conditions, which resemble that for hybrid systems without memory. 

\begin{thm}
\label{jlthm:LK}
Let $\HMD$ be a hybrid system in $\M^{\Delta}$ and let $\W\subset\Real^n$ be a closed set. Suppose that $\F$ satisfies the following local boundedness assumption: for each $b>0$, there exists some $l>0$ such that $\abs{f}\le l$ for all $f\in \F(\phi)$ and all $\phi\in\MD\cap \C$ with $\norm{\phi}_\W\le b$. If there exists a functional $V:\,\MD\ra\Real_{\ge 0}$, $\K_\infty$ functions $\alpha_i$ ($i=1,2$), and a positive definite and continuous function $\alpha_3:\,\Real_{\ge 0}\ra\Real_{\ge 0}$ such that the following hold:
\begin{enumerate}[(i)]
\item $\alpha_1(\abs{\phi(0,0)}_{\W})\le V(\phi)\le \alpha_2(\norm{\phi}_\W)$ for all $\phi\in \C\cup \D\cup\G^+(\D)$;
\item $D^+V(\phi)\le -\alpha_3(\abs{\phi(0,0)}_{\W})$ for all $\phi\in \C$;
\item $V(\phi_g^+)-V(\phi)\le -\alpha_3(\abs{\phi(0,0)}_{\W})$ for all $\phi\in\D$ and $g\in\G(\phi)$, \end{enumerate}
then $\W$ is $\KL$ pre-asymptotically stable for $\HM$.
\end{thm}

\begin{proof}
Let $x\in\SHMD(\phi)$ for some $\phi\in\MD$. Pick any $(t,j)\in\dom_{\ge 0} (x)$ and $(s,k)\in\dom_{\ge 0} (x)$ such that $(s,k)\preceq(t,j)$. Here $(s,k)\preceq(t,j)$ means $s\leq t$ and $k\leq j$. Let $s=:t_k\le t_{k+1}\le \cdots\le t_{j+1}:=t$ satisfy
$$
\dom_{\ge 0} (x)\cap ([s,t]\times\set{k,\cdots,j})= \bigcup_{i=k}^j[t_i,t_{i+1}]\times{i}.
$$
For each $i\in\set{0,\cdots,j}$ and almost all $\theta\in[t_i,t_{j+1}]$, we have $\A_{[\theta,i]}x\in\C$ and
$$
D^+V(\A_{[\theta,i]}x)\le -\alpha_3(\abs{x(\theta,i)}_{\W}),
$$
from condition (ii). Integrating this gives
\begin{equation}\label{eq:cts}
V(\A_{[t_{i+1},i]}x)-V(\A_{[t_{i},i]}x)\le -\int_{t_{i}}^{t_{i+1}}\alpha_3(\abs{x(\theta,i)}_{\W})d\theta,
\end{equation}
for all $i\in\set{k,\cdots,j}$. Moreover, for each $i\in\set{0,\cdots,j}$, we have $\A_{[t_i,i-1]}x\in\D$ and
\begin{equation}\label{eq:dts}
V(\A_{[t_{i},i]}x)-V(\A_{[t_{i},i-1]}x)\le -\alpha_3(\abs{x(t_i,i-1)}_{\W}),
\end{equation}
for all $i\in\set{k,\cdots,j}$. Combining (\ref{eq:cts}) and (\ref{eq:dts}) gives
\begin{align*}
&V(\A_{[t,j]}x)-V(\A_{[s,k]}x)\\
&\le -\sum_{i=k}^j\int_{t_i}^{t_{i+1}}\alpha_3(\abs{x(\theta,i)}_{\W})ds-\sum_{i=k+1}^j\alpha_3(\abs{x(t_i,i-1)}_{\W}).
\end{align*}
Since $\alpha_3$ is positive definite, the above inequality shows that $V(\A_{[t,j]}x)$ is non-increasing. Furthermore, letting $(s,k)=(0,0)$, together with condition (i), implies that
\begin{align}\label{eq:uniform}
\alpha_1(\abs{x(t,j)}_{\W})&\le V(\A_{[0,0]}x)\le V(\A_{[0,0]}x)\notag\\
&\le\alpha_2(\abs{x(0,0)}_{\W})\le\alpha_2(\norm{\phi}_{\W}).
\end{align}

Pick any $\eps>0$ and $\eta>0$. Let $\delta=\alpha^{-1}_2\circ\alpha_1(\eps)$. Then it follows from (\ref{eq:uniform}) that, if $\norm{\phi}_{\W}\le\delta$, then $\abs{x(t,j)}_{\W}\le\eps$ for all $(t,j)\in\dom x$.

Fixed any $\eta>0$ and consider any solution $x$ to $\HMD$ with $\norm{\phi}_{\W}\le \eta$. It follows from the above inequality that $V(x(t,j))\le \alpha_2(\eta)$ for all $(t,j)\in\dom x$. Furthermore, from the boundedness assumption on $\F$, there exists a constant $L$ such that, for all $j$ and almost all $t$ such that $(t,j)\in\dom_{\ge 0}(x)$, we have $\abs{\dot{x}(t,j)}\le L$. We further show that for all $\eps>0$, there exists some $T>0$ such that, for all $(t,j)\in\dom x$ with $t+j\ge T$, 
\begin{equation}\label{eq:attraction2}
V(x(t,j))\le \eps. 
\end{equation}

Suppose that $x$ to $\HMD$ with $\norm{\A_{[0,0]}x}_{\W}\le \eta$ is such that $\norm{\A_{[t,j]}x}_{\W}\ge\delta$ for all $(t,j)\in\dom_{\ge 0} (x)$ with $t+j\le T$. We can pick up a sequence of pairs $(t_k,j_k)\in\dom_{\ge 0} (x)$ such that
$
\abs{x(t_k,j_k)}_{\W}\ge\delta.
$
It follows that either there exists an interval $[t_k,t_k+\frac{\delta}{2L}]$ such that $[t_k,t_k+\frac{\delta}{L}]\times\set{j_k}\subset\dom_{\ge 0}(x)$ or there exists $t_k'\in [t_k,t_k+\frac{\delta}{L}]$ such that both $(t_k',j_k)\in\dom_{\ge 0}(x)$ and $(t_k',j_k+1)\in\dom_{\ge 0}(x)$ hold. In the first case, we have
$
\abs{x(s,j_k)}_{\W}\ge\frac{\delta}{2}
$
for all $s\in\big[t_k,t_k+\frac{\delta}{2L}\big].$ 
Consequently, from condition (ii), we have
\begin{equation}\label{eq:flow}
V(\A_{[t_k+\frac{\delta}{2L},j_k]}x)-V(\A_{[t_k,j_k]}x)\le -\alpha_3\Big(\frac{\delta}{2}\Big)\frac{\delta}{2L}.
\end{equation}
In the latter case, we have $\abs{x(t_k',j_k)}_{\W}\ge\frac{\delta}{2}$ and it follows from condition (iii) that
\begin{equation}\label{eq:jump}
V(\A_{[t_k',j_k+1]}x)-V(\A_{[t_k',j_k]}x)\le -\alpha_3\Big(\frac{\delta}{2}\Big).
\end{equation}
By taking a large $L$, if necessary, we can assume that $\frac{\delta}{2L}\le 1$.

Let $k_0$ be the smallest integer such that
$$\alpha_2(\eta)-\alpha_3\Big(\frac{\delta}{2}\Big)\frac{\delta}{2L}(k_0-1)<0.$$
Pick some $T>k_0(\Delta+2)$. Unless we have $\sup\set{t+j:\,(t,j)\in\dom_{\ge 0}(x)}<T$, we can pick the sequence of pairs $(t_k,j_k)\in\dom_{\ge 0} (x)$ to satisfy that $(t_k+\frac{\delta}{2L},j_k)\preceq(t_{k+1},j_{k+1})$ and $(t_k',j_k+1)\preceq(t_{k+1},j_{k+1})$ for all $k\le k_0-1$. Since $V(\A_{[t,j]}x)$ is non-increasing, it follows from combining (\ref{eq:flow}) and (\ref{eq:jump}) that
\begin{align*}
V(\A_{[t_{k_0},j_{k_0}]}x)&\le V(\A_{[0,0]}x)-\alpha_3\Big(\frac{\delta}{2}\Big)\frac{\delta}{2L}(k_0-1)\\
&\le\alpha_2(\eta)-\alpha_3\Big(\frac{\delta}{2}\Big)\frac{\delta}{2L}(k_0-1)<0,
\end{align*}
which is a contradiction. In view of (\ref{eq:uniform}) and (\ref{eq:attraction2}), the rest of the proof is similar to that for Theorem \ref{thm:stability}.

\end{proof}

\begin{rem}
A special version of Lyapunov--Krasovskii theorem for hybrid systems with delays was proved in \cite[Proposition 1]{banos2014delay} in the context of reset control systems, where the reset map acts only on the controller state and the emphasis on the delay is in continuous time. Moreover, a dwell-time condition was assumed there. Theorem \ref{jlthm:LK} covers the general case of hybrid systems with delays, where the delay can be in both continuous and discrete time.
\end{rem}

We use a simple example to illustrate the above theorem.

\begin{ex}[Time-delay systems with jumps]
Consider a hybrid system $\HMD=(\C,\F,\D,\G)$ with
$$
\F(\psi):=\begin{bmatrix}a\phi(0,0)+b\phi(-r,k(r)),\\
1
\end{bmatrix}
$$
$$
\G(\psi):=\begin{bmatrix}\rho\phi(0,0),\\
0
\end{bmatrix}
$$
$$
\C:=\set{\psi=(\phi,\tau)\in\MD:\,\tau(0,0)\in [0,\delta]},
$$
$$
\D:=\set{\psi=(\phi,\tau)\in\MD:\,\tau(0,0)=\delta},
$$
where $a$, $b$, and $\rho$ are scalar constants and $k(s)=\max\set{k:\,(s,k)\in\dom \phi}$.
\end{ex}

Let $\W:=\set{0}\times[0,\delta]$. Since $\psi=(\phi,\tau)$ and $\tau\in [0,\delta]$, we have, by the definitions of $\abs{\cdot}_{\W}$ and $\norm{\cdot}_{\W}$ (see Definition \ref{def:stability}), that $\abs{\psi(0,0)}_{\W}=\abs{\phi(0,0)}$ and $\norm{\psi}_{\W}=\norm{\phi}$ for all $\psi=(\phi,\tau)\in\C\cup\D\subset\MD$, where $\norm{\phi}=\sup_{(t,j)\in\dom\phi\atop -\Delta-1\le t+j\le 0}\abs{\phi(t,j)}$.

It is easy to see that the generalized hybrid system formulation above corresponds to the delay differential equation
$
\dot{x}(t)=ax(t)+bx(t-r)
$
subject to the reset map $x^+=\rho x$ every $\delta$ unit of time, where $x^+$ denotes the state after applying the reset map at $t$.

With $\psi=(\phi,\tau)\in\MD$, let
$$
V(\psi)=\abs{\phi(0,0)}^2e^{-\sigma\tau(0,0)}+\mu\int_{-r}^0\abs{\phi(s,k(s))}^2ds,
$$
where $\sigma$ and $\mu\ge 0$ are constants to be determined. It is easy to see that condition (i) in Theorem \ref{jlthm:LK} is satisfied with $\alpha_1(s)=s^2 e^{-\abs{\sigma}\delta}$ and $\alpha_2(s)=s^2 (e^{\abs{\sigma}\delta}+r\mu)$. If $\psi\in \C$, it follows that
\begin{align}
D^+V(\psi)&\le 2\phi(0,0)e^{-\sigma\tau(0,0)}(a\phi(0,0)+b\phi(-r,k(-r)))\notag\\
&\qquad-\sigma\phi^2(0,0)e^{-\sigma\tau(0,0)}\notag\\
&\qquad+\mu\phi^2(0,0)-\mu\phi^2(-r,k(-r)).\label{eq:flowV}
\end{align}
If $\psi\in \D$, it follows that
\begin{align}
V(\phi^+_g)-V(\psi)&\le (\rho-e^{-\sigma\delta})\abs{\phi(0,0)}^2,\label{eq:jumpV}
\end{align}
for $g\in\G(\psi)$. Consider two cases:\\
\textbf{(I) $a<0$ and $\rho>1$:} This corresponds to the case when the flow dynamics are stable, whereas the jump dynamics are not. Fix some $\sigma<0$. The inequality (\ref{eq:flowV}) would imply condition (ii) if
$
2ae^{-\sigma\delta}-\sigma e^{-\sigma\delta}+\mu<0
$
and
$
-(2ae^{-\sigma\delta}-\sigma e^{-\sigma\delta}+\mu)\mu>b^2e^{-2\sigma\delta}
$
hold simultaneously. Moreover, the inequality (\ref{eq:flowV}) would imply condition (iii) if
$
\rho<e^{-\sigma\delta}.
$
It is easy to check that, if $\abs{b}<a$, we can choose $\delta>0$ sufficiently large such that the above three inequalities always hold. Therefore, by Theorem \ref{jlthm:LK}, $\W$ is $\KL$ pre-asymptotically stable for $\HMD$ if $\delta$ is sufficiently large.\\
\textbf{(II) $a>0$ and $\rho<1$:} This corresponds to the case when the flow dynamics are unstable, whereas the jump dynamics are stable. Suppose $\sigma>0$. Similarly, the inequality (\ref{eq:flowV}) would imply condition (ii) if $2a-\sigma+\mu<0$ and $-(2a-\sigma+\mu)\mu>b^2$ hold simultaneously. With $\mu=2b$, we can choose $\sigma>2(a+\abs{b})$ such that the above inequality holds. Moreover, the inequality (\ref{eq:flowV}) would imply condition (iii) if
$
\rho<e^{-\sigma\delta}.
$
This can be done by choosing $\delta$ sufficiently small according to a chosen value of $\sigma$. Therefore, by Theorem \ref{jlthm:LK}, $\W$ is $\KL$ pre-asymptotically stable for $\HMD$ if $\delta$ is sufficiently small.

\section{Conclusions}

In this note, we have investigated asymptotic stability of hybrid systems with memory via generalized solutions. While the motivation for considering generalized solutions lies in the needs to address robustness of asymptotic stability and well-posedness for such systems, the focus of this note is on asymptotic stability analysis using Lyapunov-based methods. We have established two sets of Lyapunov-based sufficient conditions for the asymptotic stability of hybrid systems with memory, one based on Lyapunov--Razumikhin techniques and the other Lyapunov--Krasovskii functionals. We have demonstrated the effectiveness of these techniques using two examples.

\ifCLASSOPTIONcaptionsoff
  \newpage
\fi

\bibliographystyle{IEEEtran}
\bibliography{hybrid}
\end{document}